\documentclass[hidelinks,onefignum,onetabnum]{siamart250211}

\usepackage{amsmath,amsfonts,amssymb}
\usepackage{xcolor,bm,url}
\usepackage{booktabs,array}
\usepackage{geometry}
\usepackage{tikz}
\usepackage{algorithmic}
\usepackage{comment}
\usepackage{etoolbox}
\usepackage{xspace}
\usepackage{amsmath}
\usepackage{enumitem}

\newsiamremark{remark}{Remark}
\newsiamremark{assumption}{Hypothesis}
\crefname{assumption}{Hypothesis}{Hypotheses}

\usepackage{glossaries}
\newacronym{nlp}{NLP}{nonlinear programming}
\newacronym{kkt}{KKT}{Karush-Kuhn-Tucker}
\newacronym{ipm}{IPM}{interior-point method}
\newacronym{cpu}{CPU}{central processing units}
\newacronym{gpu}{GPU}{graphics processing units}
\newacronym{mpc}{MPC}{model predictive control}
\newacronym{ac}{AC}{alternating current}
\newacronym{opf}{OPF}{optimal power flow}
\newacronym{hpc}{HPC}{high-performance computing}
\newacronym{pcg}{PCG}{projected conjugate gradient}
\newacronym{alm}{ALM}{augmented Lagrangian method}
\newacronym{der}{DER}{distributed energy resource}
\newacronym{derms}{DERMS}{distributed energy resource management system}
\newacronym{bess}{BESS}{battery energy storage system}
\newacronym{iso}{ISO}{independent system operator}
\newacronym{simd}{SIMD}{single instruction, multiple data}
\newacronym{spd}{SPD}{symmetric positive definite}
\newacronym{sqd}{SQD}{symmetric quasi-definite}
\newacronym{sqp}{SQP}{sequential quadratic programming}
\newacronym{ncl}{NCL}{nonlinearly constrained Lagrangian}
\newacronym{bcl}{BCL}{bound-constrained Lagrangian}
\newacronym{licq}{LICQ}{linear independence constraint qualifications}
\DeclareMathOperator{\diag}{\mathrm{diag}}

\DeclareMathOperator{\Inertia}{\mathrm{In}}
\newcommand{\tiptop}{\top\!}
\newcommand{\Lagr}{\mathcal{L}}
\newcommand{\R}{\mathbb{R}}
\newcommand{\half}{\frac{1}{2}}

\newcommand{\bmat}[1]{\begin{bmatrix}#1\end{bmatrix}} 

\newcommand{\ldlt}{$\mathrm{LDL^T}\,$}
\newcommand{\lblt}{$\mathrm{LBL^T}\,$}

\patchcmd{\SetTagPlusEndMark}{$}{}{}{}
\patchcmd{\SetTagPlusEndMark}{$}{}{}{}

\newcommand{\minim}{\mathop{\mathrm{minimize}}}
\newcommand{\minimize}[1]{{\displaystyle\minim_{#1}}}
\newcommand{\norm}[1]{\|#1\|}
\newcommand{\T}{^T\!}

\newcommand{\BCk} {\mbox{BC$_k$}\xspace}

\newcommand{\NCk} {\mbox{NC$_k$}\xspace}

\newcommand{\problem}[4]{\fbox
   {\begin{tabular*}{0.84\textwidth}
    {@{}l@{\extracolsep{\fill}}l@{\extracolsep{6pt}}%
        l@{\extracolsep{\fill}}c@{}}
      #1 & $\minimize{#2}$ & $#3$ & $ $ \\[5pt]
         & subject to  & $#4$ & $ $
    \end{tabular*}}}

\title{MadNCL: A GPU Implementation of 
Algorithm NCL for Large-Scale, Degenerate Nonlinear Programs}

\author{Alexis Montoison \and François Pacaud \and Michael Saunders \and Sungho Shin \and Dominique Orban}
\date{\today}

\begin{document}
\maketitle

\begin{abstract}
We present a GPU implementation of Algorithm NCL, an augmented Lagrangian method
for solving large-scale and degenerate nonlinear programs. Although
interior-point methods
and sequential quadratic programming are widely used for solving nonlinear programs,
the augmented Lagrangian method is known to offer superior robustness against constraint degeneracies and can rapidly detect infeasibility.
We introduce several enhancements to Algorithm NCL, including fusion of the inner and outer loops and use of extrapolation steps, which improve both efficiency and convergence stability.
Further,  NCL  has the key advantage of being well-suited for GPU
architectures because of the regularity of the
KKT systems
provided by quadratic penalty terms. In particular, the NCL
subproblem
formulation allows the KKT 
systems to be naturally expressed as either
stabilized or condensed KKT systems, whereas 
the interior-point approach
requires aggressive reformulations or relaxations to make it suitable for GPUs. 
Both systems can be efficiently solved on GPUs 
using sparse \ldlt factorization
with static pivoting, as implemented in NVIDIA cuDSS.
 Building on these advantages, we examine the KKT
systems arising from NCL 
subproblems.
We present an optimized GPU 
implementation of Algorithm NCL 
by leveraging MadNLP as an interior-point subproblem solver and
utilizing the stabilized and condensed formulations of the KKT 
systems
for computing Newton steps. Numerical experiments on various large-scale and
degenerate NLPs, including optimal power flow, COPS benchmarks,
and security-constrained optimal power flow, demonstrate that MadNCL operates efficiently on
GPUs 
while effectively managing problem degeneracy,
including MPCC constraints.
\end{abstract}

\begin{keywords}
  nonlinear programming,
  augmented Lagrangian method,
  interior-point methods,
  constraint qualifications,
  degeneracy,
  KKT systems,
  MPCC constraints,
  graphics processing units
\end{keywords}

\begin{AMS}
    65K05, 
    90C30, 
    49M37, 
    90C06 
\end{AMS}

\section{Introduction}
We consider the \gls{nlp} problem
\[\label{eqn:nlp}
   \problem{NLP}{x \in \mathbb{R}^n}
                {\phi(x)}
                {c(x)=0, \quad \ell \leq x \leq u,}
\]
where $\phi: \mathbb{R}^n \to \mathbb{R}$ is a smooth nonlinear objective function and $c: \mathbb{R}^n \to \mathbb{R}^m$ is a vector of smooth nonlinear constraint functions.
We assume that $\phi$ and $c$ are twice continuously differentiable and their derivatives are accessible, thereby enabling the use of second-order algorithms.
Our focus is on an
augmented Lagrangian method 
for large-scale instances of \gls{nlp} problems on \gls{gpu} hardware.
In particular, we are interested in problems with degenerate constraints, i.e., for which a standard constraint qualification condition does not hold.

\paragraph{Augmented Lagrangian method}
The \gls{alm} is one of the major \gls{nlp} algorithms \cite{nocedal_numerical_2006}. While \gls{sqp} and interior-point methods gained prominence in the 1980s because of their fast convergence properties \cite{nocedal_numerical_2006}, \gls{alm}, first devised in the 1960s \cite{hestenes1969multiplier,rockafellar1974augmented}, has the key advantage of robustly handling degenerate optimization problems and is capable of detecting infeasibility quickly \cite{chiche2014augmented}. We refer to \cite{bertsekas1976multiplier,bertsekas2014constrained,birgin2014practical} for comprehensive surveys on \gls{alm}.
The augmented Lagrangian penalty can be interpreted as a quadratic regularization in the dual~\cite{rockafellar1974augmented}.
Hence, \gls{alm} subproblems are nondegenerate regardless of the reformulation being used.
This property enables \gls{alm} to achieve superior robustness, especially for NLP problems with redundant constraints. Because of its robustness, \gls{alm} has witnessed a resurgence in the 2010s, coinciding with new methodological breakthroughs \cite{gill2012primal,curtis2015adaptive,izmailov2015combining,ma2018stabilized}.
It has been proven to exhibit fast local convergence under minimal assumptions \cite{fernandez2012local}.
Software-wise, \gls{alm} has been utilized effectively across a wide range of solver implementations, including the classical LANCELOT \cite{CGT91,CGT92} and MINOS \cite{MINOSGAMS} solvers.  Most of these improvements have been aggregated into the solver ALGENCAN \cite{andreani2008augmented}, which now offers competitive performance relative to LANCELOT.
In parallel, there has been recent interest in merging the augmented Lagrangian with \gls{ipm}~\cite{armand2017mixed,kuhlmann2018primal}.

There are two main approaches regarding the subproblem formulations within \gls{alm}, which greatly affect the numerical treatment of the subproblem solution procedures. The classical implementations (e.g., LANCELOT) are based on \gls{bcl} subproblems:
\[
   \problem{\BCk}{x \in \R^n}
   {\phi(x) - y_k\T c(x) + \half \rho_k \norm{c(x)}^2}
   {\ell \le x \le u,}
\]
where $y_k \in \R^m$ is an estimate of Lagrange multipliers for the equality constraints $c(x)=0$, $\rho_k  > 0$ is a penalty parameter, and $k=1,2,\dots, p$ is an iteration counter. In \BCk, 
the nonlinear constraints are included in the quadratic penalty term of the objective function, and the bound constraints are treated directly. This structure provides an advantage in applying active-set approaches, such as projected gradient or projected Newton methods, to solve the subproblems. Further, 
only the objective changes between two consecutive \gls{alm} iterations, enabling effective warm-starting of active-set methods. 
A disadvantage of the \gls{bcl} formulation is that the quadratic penalty term can lead to fully dense Lagrangian Hessians when a dense row exists in the Jacobian of $c(x)$.
Large dense Hessians can be challenging to handle unless their internal structure is exploited.

In Algorithm NCL \cite{ma2018stabilized}, subproblem \BCk is formulated as 
\[
   \problem{\NCk}{x \in \R^n,\,r \in \R^m}
                 {\phi(x) + y_k\T r
                          + \half \rho_k \norm{r}^2}
                 {c(x) + r = 0, \quad \ell \le x \le u,}
\]
where the added free variables $r$ render the nonlinear constraints linearly independent.
Although \NCk is mathematically equivalent to \BCk, its solution needs a different numerical implementation.
In particular, it contains nonlinear constraints and many new variables $r$, requiring an \gls{ipm} solver. The key advantage of the NCL formulation arises from its more flexible handling of constraints: the new variables $r$ play the role of a natural regularizer and enable solution of problems with constraint degeneracy (violation of LICQ).

The performance of an \gls{alm} solver depends directly on the algorithms used to solve the subproblems.
The outer iterations of \gls{alm} are typically simple factorization-free operations, and thus, most of the computational effort is spent on solving the subproblems.
In NCL, the \gls{ipm} subproblem solver for \NCk spends most of its time solving \gls{kkt} systems to compute search directions.
Thus, efficient subproblem solution strategies are key to the efficient implementation of NCL. 

\paragraph{\gls{gpu} implementations of optimization solvers}
As \gls{gpu}s become more prevalent in scientific computing, there is a stronger incentive to leverage \gls{gpu}-accelerated routines inside optimization solvers.
Current optimization solvers are largely based on algorithms developed in the 1980s and 1990s, and were not designed to take advantage of \gls{gpu} architecture.
Consequently, most optimization solvers are sequential in nature, utilizing parallel operations only at the linear algebra level (e.g., multi-thread parallelism within BLAS routines or in the linear solver used to compute each Newton direction).
Given these limitations, the last several years have seen a surge in efforts to harness \gls{gpu}s to accelerate the solution of large-scale optimization problems, resulting in several notable breakthroughs.

Recent developments in mathematical optimization on \gls{gpu}s can be classified into two concurrent threads. The first thread employs first-order or factorization-free methods, with a focus on convex optimization problems. We refer to the recent \gls{gpu} implementation of PDLP (Primal-Dual Hybrid Gradient for LP) to solve large-scale LPs~\cite{lu2023cupdlp,lu2025cupdlpx} and the ADMM algorithm implemented in OSQP~\cite{schubiger2020gpu}. These methods rely on efficient factorization-free implementations to leverage parallelism. For example, PDLP can be implemented using sparse matrix-vector multiplication, which is amenable to parallelization with efficient implementations on \gls{gpu}s (e.g., via cuSPARSE or rocSPARSE libraries) \cite{lu2023cupdlp}. Similarly for projected gradient-based implementations of operator-splitting methods
\cite{schubiger2020gpu}.

The second thread of research focuses on second-order methods with direct sparse solvers, which have gained popularity for nonconvex NLPs (where first-order or factorization-free approaches are not effective). In this approach, \gls{ipm} is used as a skeleton, and various linear algebra tricks are employed to handle the associated \gls{kkt} systems. Until recently, the sparse factorizations available on \gls{gpu}s were notoriously slow~\cite{tasseff2019exploring}. Consequently, solvers for general NLPs were not available, and domain-specific implementations (e.g., exploiting a nonsingular block within the constraint Jacobian) have been investigated~\cite{pacaud2024accelerating}. This has changed with the recent release of NVIDIA cuDSS, an efficient general-purpose sparse linear solver for \gls{gpu}s.
Although cuDSS cannot be directly utilized within existing NLP solvers because of inefficiencies in handling indefinite systems, two reformulation strategies can be employed to enforce \gls{gpu}-amenable structures in the \gls{kkt} system: hybrid \gls{kkt} system~\cite{regev2023hykkt} and lifted \gls{kkt} system~\cite{shin2024accelerating}. With these reformulations, the \gls{kkt} systems can be solved effectively on \gls{gpu}s, making the implementation of a general-purpose \gls{gpu} \gls{ipm} solver practical.
An example is MadNLP~\cite{shin2024accelerating},
which we use as a stand-alone IPM NLP solver, and as
an NCL subproblem solver within MadNCL (our implementation of Algorithm NCL).

Although the lifted and hybrid \gls{kkt} system reformulation strategies enable the solution of \gls{kkt} systems on \gls{gpu}s, they cannot achieve the same degree of robustness as the classical CPU implementations of \gls{ipm} solvers, such as Ipopt or KNITRO. As a consequence, the default optimality tolerance for the \gls{gpu} version of MadNLP is relaxed to $10^{-6}$.
This is significantly looser than the $10^{-8}$ tolerance in mature \gls{ipm}s. The primary reason for this compromised solver robustness stems from the significantly increased condition number of the \gls{kkt} system resulting from the condensation processes used within these approaches \cite{pacaud2024condensed}. This 
is currently recognized as a fundamental limitation of \gls{gpu}-accelerated \gls{ipm} solvers.

\paragraph{Goal}
Our main research question is:
\begin{quote}
  \emph{Can a \gls{gpu} implementation of
  Algorithm NCL
  overcome the limitations of current \gls{ipm} \gls{gpu} implementations in terms of robustness and achievable optimality tolerance while maintaining similar speed?}
\end{quote}
We aim to answer this question by (i) developing efficient strategies to
solve the \gls{kkt} systems arising from
NCL subproblems on
\gls{gpu}s, and (ii) demonstrating that MadNCL, the \gls{gpu} implementation of NCL,
can achieve competitive performance relative to existing \gls{gpu}
implementations of IPM solvers while maintaining robustness, particularly for degenerate problems.
We develop these strategies by adapting the recent condensed-space method of
\cite{pacaud2024condensed} to handle the \gls{kkt} systems within NCL
subproblems. The condensed \gls{kkt} systems within IPMs and NCL subproblems share substantial
similarity, allowing us to leverage existing \gls{gpu}-accelerated linear
algebra kernels implemented in MadNLP.
Further, the structure of subproblem \NCk leads naturally to the formulation of sparse \emph{condensed KKT} or
sparse \emph{stabilized KKT} matrices, which can be factorized using static pivoting, respectively by the Cholesky and
\ldlt algorithms implemented in \gls{gpu} solvers like cuDSS. In this
article, we provide an optimized \gls{gpu} implementation of \gls{alm} based on the NCL subproblem formulation, which we refer to as
MadNCL, by leveraging the \gls{gpu}-accelerated solver
MadNLP~\cite{shin2024accelerating} as an IPM subproblem solver and
utilizing \emph{condensed} and \emph{stabilized} formulations of the \gls{kkt} systems for computing
Newton steps. Our numerical results show that when running on the \gls{gpu},
NCL offers a more robust alternative to the condensed-space method of
\cite{pacaud2024condensed}. In particular, we demonstrate that although
\gls{alm}s can be slightly slower than condensed-space IPM
approaches, they are significantly more robust at solving degenerate problems,
overcoming the limitations of existing \gls{gpu}-accelerated IPM solvers.

\paragraph{Contributions}
Our main contributions follow.
\begin{itemize}[leftmargin=*]
\item
  We provide for the first time a \gls{gpu} \gls{alm} implementation for solving large-scale \glspl{nlp}.
  Notably, the previous implementations of Algorithm NCL used the \gls{ipm} solvers Ipopt and KNITRO to solve subproblems \NCk \cite{ma2018stabilized,ma2021julia}.
  Instead, we use MadNLP~\cite{shin2024accelerating}, a \gls{gpu} \gls{ipm} implementation with a filter line search algorithm~\cite{wachter2006implementation}.
  MadNCL has full control over MadNLP's internals, enabling it to exploit the structure of the KKT systems in the subproblems \NCk for more efficient linear algebra operations.
\item
  We present two different \gls{kkt} system formulations, called the \emph{stabilized KKT system} ($K_{2r}$)
  and the \emph{condensed KKT system} ($K_{1s}$), adopting the naming convention introduced in~\cite{ghannad-orban-saunders-2022}.
  We show that both systems can be factorized efficiently on \gls{gpu}s without the use of numerical pivoting.
  As a consequence, the $K_{2r}$ and $K_{1s}$ systems can be solved on the GPU by NVIDIA cuDSS.
  Furthermore, these reformulations inherit 
  established properties
  of the commonly used augmented KKT system ($K_2$) formulation, thereby guaranteeing
  descent directions for the NCL subproblems without introducing additional computational overhead.
\item
    Numerically, we demonstrate that NCL can effectively take advantage of \gls{gpu} parallelism and exhibit superior robustness.
    In particular, NCL can be implemented seamlessly on a \gls{gpu}, provided a GPU-based IPM solver like MadNLP is available, and the resulting MadNCL offers competitive performance relative to other \gls{gpu} optimization solvers~\cite{pacaud2024condensed}.
    Further, we show that MadNCL is able to solve large real-world degenerate problems, such as large-scale mathematical programs with complementarity constraints (MPCC).
\end{itemize}

\paragraph{Outline}
The remainder of this paper is organized as follows.
\Cref{sec:alm} provides a brief overview of \gls{alm} and the formulation of NCL subproblems.
\Cref{sec:kkt} describes the \gls{kkt} systems arising within NCL subproblems and their condensed forms.
\Cref{sec:num:implementation} outlines implementation details of MadNCL, our \gls{gpu} implementation of \gls{alm} based on the NCL subproblem formulation.
\Cref{sec:num} presents numerical results for MadNCL on various degenerate \glspl{nlp}.
\Cref{sec:conc} concludes and discusses future work.

\section{Augmented Lagrangian Method}\label{sec:alm}
By design, \gls{alm} is a two-level method: the penalty parameter and Lagrange multipliers are updated
during an \emph{outer loop} (see \S\ref{sec:ncl:outer}), whereas the subproblems \BCk or \NCk are
solved in an \emph{inner loop} (see \S\ref{sec:ncl:inner}). Our implementation MadNCL
merges these two loops to improve the overall performance, following \cite{armand2017mixed}.

\subsection{\gls{alm} outer loop} \label{sec:ncl:outer}
\gls{alm} solves the augmented Lagrangian subproblem with a fixed penalty parameter $\rho_k > 0$ and a multiplier estimate $y_k \in \mathbb{R}^m$.
Updating either the multipliers or the penalty parameter depends on the observed primal infeasibility just after completion of the current inner loop (solution of the subproblem).

Upon obtaining a solution $(x_{k+1}, r_{k+1})$ for the augmented Lagrangian subproblem, \gls{alm} updates the multipliers $y_k$ if the primal infeasibility is below a specified accuracy level $\eta_k$; otherwise, it increases the penalty parameter $\rho_k$.
The following is a commonly used update rule:
\begin{equation}
  \label{eq:auglagupdate}
(y_{k+1},~\rho_{k+1}) =
\begin{cases}
(y_k - \rho_k c(x_{k+1}),~\rho_k) & \text{if } \|c(x_{k+1})\|_{\infty} \leq \eta_k, \\
(y_k,~10\rho_k) & \text{otherwise.}
\end{cases}
\end{equation}
In this update rule, the multipliers $y_k$ are updated by subtracting the primal infeasibility $c(x_{k+1})$ scaled by the penalty parameter $\rho_k$ whenever the primal infeasibility is sufficiently small. This can be interpreted as a gradient ascent step in the dual space, with the step size determined by the penalty parameter $\rho_k$~\cite{rockafellar1974augmented}.
If the primal infeasibility is not sufficiently small, the penalty parameter $\rho_k$ is increased by a factor of $10$ to encourage the next subproblem to exhibit less infeasibility.
For nonconvex problems, sufficiently large penalty parameters must be employed to ensure that the subproblem can recover the exact solution when the multiplier estimates are accurate \cite{nocedal_numerical_2006}.
While traditional \gls{alm} implementations utilize a projected Newton method~\cite{CGT92,nocedal_numerical_2006} to solve \BCk, the solution of \NCk requires the use of an IPM solver, namely MadNLP here.

\subsection{\gls{ipm} inner loop} \label{sec:ncl:inner}

The \gls{ncl} subproblem formulation~\cite{ma2018stabilized} adapts the classical BCL formulation by introducing free variables $r \in \mathbb{R}^m$ and equality constraints. This reformulation yields the \emph{nonlinearly constrained subproblem} \NCk.
Although \NCk
is much larger than \BCk,
its structure is more amenable to efficient solution using an \gls{ipm}: the variables $r$ render the constraints linearly independent. Further, as the parameters $\rho_k$ and $y_k$ appear only in the objective, it is easier to warm-start an \gls{ipm} to solve \NCk with new parameters $\rho_{k+1}$ and $y_{k+1}$ derived from the usual update rule~\eqref{eq:auglagupdate}, as noted in \cite{ma2021julia}.

For fixed parameters $(y_k, \rho_k)$, let $y \in \mathbb{R}^{m}$ be the multipliers associated with the equality constraints in \NCk, and $(z_l, z_u) \in \mathbb{R}^{n} \times \mathbb{R}^{n}$ be the nonnegative multipliers related to the bound constraints on~$x$. The Lagrangian for \NCk is
\begin{equation}
    \label{eq:lagrangianncl}
    \Lagr(x, r, y, z_l, z_u) = \phi(x) + y_k^\top r + \tfrac{1}{2} \rho_k \| r \|^2 - y^\top (c(x) + r) - z_l (x-\ell) - z_u (u-x),
\end{equation}
and the associated \gls{kkt} conditions are
\begin{equation}
  \label{eq:kktncl}
  \begin{aligned}
     \nabla \phi(x) - \nabla c(x)^\top y - z_l + z_u &= 0 \\
     y_k + \rho_k r - y &= 0 \\
     c(x) + r &= 0 \\
     0 \le x-\ell \perp z_l &\ge 0 \\
     0 \le u-x \perp z_u &\ge 0.
  \end{aligned}
\end{equation}
Algorithm \gls{ncl}~\cite{ma2018stabilized} applies Newton's method to \eqref{eq:kktncl}, globalized here with a filter line search~\cite{wachter2006implementation}. Once a primal solution $(x_{k+1}, r_{k+1})$ is found, parameters $y_k$ and $\rho_k$ are updated based on rule \eqref{eq:auglagupdate}, and
we proceed by solving a new subproblem NC$_{k+1}$.

\subsection{Fusion of \gls{alm} outer loop and \gls{ipm} inner loop} \label{sec:ncl:madncl}

MadNCL fuses the outer and inner loops described 
in \S\ref{sec:ncl:outer} and in \S\ref{sec:ncl:inner}. In addition,
MadNCL uses an extrapolation step to converge asymptotically at a superlinear rate~\cite{armand2017mixed,dussault1995numerical}.
The complete algorithm is detailed in Algorithm~\ref{algo:nclipm}.

\begin{algorithm}[!ht]
  \caption{MadNCL}
  \label{algo:nclipm}
  \begin{algorithmic}
  \STATE{Initialize variable $x_0$.}
  \STATE{Set primal and dual tolerances $(\eta_\star, \omega_\star)$.}
  \STATE{Set $\gamma = 0.05$, $\tau = 1.99$, $\mu_{\text{fac}} = 0.2$, $\rho_{\max} = 10^{14}$, $\theta = 0.5$.}
  \STATE{Set initial parameters $\rho_0 \gets 100$ , $\mu_0 \gets 0.1$.}
  \STATE{Compute initial dual variable $y_0$ using least-squares.}
  \FOR{$k=1,2,\dots$}
    \STATE{Solve $\nabla_w F_k(w_k) d_k + F_k(w_k) = 0$ and set $w_k^+ = w_k + \alpha_k d_k$.}
    \IF{$ \|F_k(w_{k}^+) \|_\infty \leq \theta \|F_k(w_k) \|_\infty + 10 \alpha_k^{0.2} \mu_k$}
    \STATE{Set $w_{k+1} = w_k^+$}
    \ELSE
    \STATE{Find $w_{k+1}$ satisfying $\| F_k(w_{k+1}) \|_\infty \leq \omega_k$ using MadNLP.}
    \ENDIF
    \IF{$\|r_{k+1}\|_\infty \leq \eta_k$}
    \STATE{$y_{k+1} \gets y_k + \rho_k r_{k+1}$}
    \STATE{$\mu_{k+1} \gets \min\{ (\mu_k)^\tau, \mu_{fac} \times \mu_k \}$}
    \STATE{$\eta_{k+1} \gets \min\{(\mu_{k+1})^{1.1}, 0.1 \times \mu_k \}$}
    \STATE{$\omega_{k+1} \gets 100 \times (\mu_{k+1})^{1 + \gamma}$}
    \STATE{$\rho_{k+1} \gets \rho_k$}
    \ELSE
    \STATE{$\rho_{k+1} \gets \min\{\rho_{\max}, 10 \rho_k \}$}
    \STATE{$(y_{k+1}, \mu_{k+1})  \gets (y_k, \mu_k)$}
    \STATE{$(\eta_{k+1}, \omega_{k+1}) \gets (\eta_k, \omega_k)$}
    \ENDIF
    \IF{$\|r_{k+1}\|_\infty \leq \eta_\star$ \text{~and~} $\|\nabla f(x_{k+1}) - \nabla c(x_{k+1})^\top y_{k+1} \|_\infty \leq \omega_\star$}
    \STATE{Solution is locally optimal, stop}
    \ENDIF
    \IF{$\rho \geq \rho_{max}$ \text{~and~} $\|r_{k+1}\|_\infty > \eta_{\star}$}
    \STATE{Problem is locally infeasible, stop}
    \ENDIF
  \ENDFOR
  \end{algorithmic}
\end{algorithm}

\paragraph{Inner iteration}
For a given barrier parameter $\mu_k > 0$, multiplier estimate $y_k$, and penalty $\rho_k$,
\gls{ipm}
reformulates the KKT equations~\eqref{eq:kktncl} using a homotopy continuation method.
Let $w := (x, r, y, z_l, z_u)$ denote the vector of primal-dual variables. For primal variables in the strict interior, where $(x - \ell,~u - x) > 0$, the inner iterations solve the following system for $w$:
\begin{equation}
  \label{eq:kktnclipm}
F(w, \rho_k, y_k, \mu_k) =
  \begin{bmatrix}
     \nabla \phi(x) - \nabla c(x)^\top y - z_l + z_u
  \\ y_k + \rho_k r - y
  \\ c(x) + r
  \\ Z_l (x - \ell) - \mu_k e
  \\ Z_u (u - x) - \mu_k e
  \end{bmatrix} \; = 0 ,
\end{equation}
with $Z_l = \diag(z_l)$ and $Z_u = \diag(z_u)$.
As \eqref{eq:kktnclipm} is a smooth system of nonlinear equations,
 the inner \gls{ipm} iterations utilize a globalized Newton method to solve it.
 For a specified tolerance $\omega_k$, the next primal-dual iterate $w_{k+1}$ solves
\begin{equation}
  \label{eq:ncl:inneriter}
  \| F(w_{k+1}, \rho_k, y_k, \mu_k) \|_\infty \leq \omega_k.
\end{equation}
To simplify notation, we define $F_k(w) := F(w, \rho_k, y_k, \mu_k)$.

\paragraph{Extrapolation step}
Once it is close to convergence, MadNCL applies an extrapolation step
to achieve fast superlinear convergence \cite{armand2013global,dussault1995numerical}.
Let $\theta \in (0, 1)$ and $\varepsilon_k > 0$.
Before solving \eqref{eq:ncl:inneriter}, the extrapolation step solves the linear system
\begin{equation}
  \label{eq:ncl:extrapolation}
  \nabla_w F_k(w_k) d_k + F_k(w_k) = 0
\end{equation}
and sets $w_k^+ = w_k + \alpha_k d_k$, with $\alpha_k$ computed
using a fraction-to-boundary rule ensuring that $w_k^+$ remains strictly feasible
(i.e., $\ell < x_k^+ < u$ and $(z_{\ell,k}^+, z_{u,k}^+) > 0$).
If
\begin{equation}
  \|F_k(w_{k}^+) \|_\infty \leq \theta \|F_k(w_k) \|_\infty + \varepsilon_k \; ,
\end{equation}
we set $w_{k+1} = w_k^+$ directly. Otherwise, we perform inner iterations
to find an iterate $w_{k+1}$ satisfying \eqref{eq:ncl:inneriter}.
Following \cite[Section 5.3]{armand2013global}, we set $\varepsilon_k = 10 \alpha_k^{0.2} \mu_k$.

In other words, if $w_k^+$ makes sufficient progress towards optimality, we discard
the inner iteration and move directly to the next outer iteration by setting
$w_{k+1} = w_k^+$. This
ensures a full Newton step close to optimality, resulting in
a superlinear rate of convergence in the final iterations~\cite{arreckx2018regularized}.

\paragraph{Outer iteration}
The outer iterations are a variant of the 
implementation in
LANCELOT, as we simultaneously update the barrier parameter with the
other augmented Lagrangian parameters~\eqref{eq:auglagupdate}. The update rules for tolerances
$\eta_k$ and $\omega_k$ are adapted from the Superb algorithm~\cite{gould2015interior}.

\section{Analysis of the \gls{kkt} systems}
\label{sec:kkt}

In this section, we analyze the linear systems arising from the NCL
subproblem formulation and demonstrate how they can be transformed into a
\emph{stabilized KKT system} and a \emph{condensed KKT system}, suitable for efficient \gls{gpu} implementation.

\subsection{Newton system}
In Algorithm~\ref{algo:nclipm} (MadNCL), the inner iterations consist in applying Newton's method to~\eqref{eq:kktnclipm} with fixed parameters $(y_k, \rho_k, \mu_k)$. Successive Newton iterations lead to a sequence of \gls{kkt} systems. In deriving these linear systems, we explicitly distinguish between equality and inequality constraints to exploit the block structure associated with slack variables. Therefore, in this section, we base our analysis on the following Nonlinear Constrained Optimization (NCO) problem:
$$
   \problem{NCO}{t}
                {f(t)}
                {c_{\mathcal{E}}(t)=0,\quad \ell_s\le c_{\mathcal{I}}(t)\le u_s,\quad \ell_t\le t\le u_t.}
$$
Compared to problem NLP, we treat equality and inequality constraints separately. This formulation can be transformed into problem NLP by setting
$$
\phi(x) = f(t), \quad c(x) = \bmat{ c_{\mathcal{E}}(t) \\ c_{\mathcal{I}}(t) - s } = 0, \quad x = \bmat{ t \\ s }, \quad \ell = \bmat{ \ell_t \\ \ell_s }, \quad u = \bmat{ u_t \\ u_s }.
$$

We now derive the various Newton systems. Assuming $x - \ell$ and $u - x$ are nonnegative, Algorithm~\ref{algo:nclipm} finds the Newton descent direction $\Delta w = (\Delta x,\,\Delta r,\,\Delta y,\,\Delta z_l,\,\Delta z_u)$ as the solution of
\begin{equation}
  \label{eq:newton_step}
  \nabla_w F(w, \rho_k, y_k, \mu_k) \Delta w = -F(w, \rho_k, y_k, \mu_k).
\end{equation}
Note that $\Delta w$ corresponds to $d_k$ in Algorithm~\ref{algo:nclipm}.
This gives
the 
linear system
\begin{equation}
  \tag{$K_3$}
  \label{eq:K3}
  \bmat{
    \phantom{-} H     & 0 & - J^\top  & \!\!\!\!-I & I \\
    \phantom{-}0      & \rho_k I & - I\phantom{^\top}  & \!\!\!\!\phantom{-}0  & 0 \\
    \phantom{-}J      & I & 0  & \!\!\!\!\phantom{-}0  & 0 \\
    \phantom{-}Z_l    & 0 & 0  & \!\!\!\!\phantom{-}X - L & 0 \\
    -Z_u   & 0 & 0  & \!\!\!\!\phantom{-}0  & U - X
  }
  \bmat{
       \Delta x \\ \Delta r \\ \Delta y \\ \Delta z_l \\ \Delta z_u
  }
  =
    - \bmat{
     \nabla \phi(x) -  J^\top y  - z_l + z_u \\
     y_k + \rho_k r - y \\
     c + r \\
     Z_l (x - \ell) - \mu_k e \\
     Z_u (u - x) - \mu_k e
  },
\end{equation}
where $c = c(x)$, $J^\top = J(x)^\top = [\nabla c_1(x)\ \cdots\ \nabla c_m(x)]$, $S = \diag(s)$, $X = \diag(x)$, $L = \diag(\ell)$, $U = \diag(u)$, $Z_l = \diag(z_l)$, $Z_u = \diag(z_u)$, and $H$ is the Hessian of the Lagrangian $\mathcal{L}(x, r, y, z_l, z_u)$ (with respect to $x$):
$$
H = \nabla^2 \phi(x) - \sum_{j=1}^{m} y_j \nabla^2 c_j(x).
$$

Eliminating
\begin{equation}
  \label{eq:K3_blocks}
  \begin{aligned}
    \Delta z_l &= -(X-L)^{-1} \bigl(Z_l\,\Delta x - \mu_k e\bigr) - z_l, \\
    \Delta z_u &= (U-X)^{-1} \bigl(Z_u\,\Delta x + \mu_k e\bigr) - z_u
  \end{aligned}
\end{equation}
leads to the equivalent symmetrized Newton system (used by default in previous implementations of NCL~\cite{ma2018stabilized,ma2021julia}):
\begin{equation}
  \label{eq:K2unreg}
  \bmat{
       \hat{H} & 0           & J^\top \\
       0       & \rho_k I    & I \\
       J       & I           & 0
  }
  \bmat{
    \phantom{-}\Delta x \\ \phantom{-}\Delta r \\ -\Delta y
  }
  =
    - \bmat{
     \nabla \phi(x) -  J^\top y -(X-L)^{-1} \mu_k e + (U-X)^{-1} \mu_k e \\
     y_k + \rho_k r - y \\
     c + r
  },
\end{equation}
where $\hat{H} = H + \Sigma$ and $\Sigma := (X - L)^{-1} Z_l + (U - X)^{-1} Z_u$ is diagonal.

To get a valid search direction $(\Delta x, \Delta r, -\Delta y)$, the filter line-search algorithm requires
that \eqref{eq:K2unreg} is nonsingular and the Hessian projected onto the null space of Jacobian $J$ is symmetric and positive definite (SPD).
This is equivalent to stating that the inertia (the number of positive, negative, and zero eigenvalues) of \eqref{eq:K2unreg} is equal to $(n+m, m, 0)$~\cite{benzi2005numerical,gould1985practical}.
The free variables $r$ guarantee that the constraint Jacobian $\bmat{ J & I }$ has full row rank, ensuring that the matrix in \eqref{eq:K2unreg} has no zero eigenvalues.
To obtain the correct inertia, inertia-correction methods add a primal regularization term $\delta_k$ to the upper-left blocks, replacing the matrix in \eqref{eq:K2unreg} by the regularized matrix
\begin{equation}
  \tag{$K_2$}
  \label{eq:K2}
  \bmat{
       \hat{H} + \delta_k I  & 0                     & J^\top \\
       0                     & (\rho_k + \delta_k) I & I \\
       J                     & I                     & 0
  } \;.
\end{equation}
The regularization parameter $\delta_k$ is increased until the desired inertia is achieved, guaranteeing that a valid search direction is found for the filter line-search algorithm. To simplify notation,
we incorporate the inertia regularization into the penalty parameter by defining
\[
\hat{\rho}_k := \rho_k + \delta_k.
\]

\paragraph{Sparse \lblt factorization}
If a sparse \lblt factorization is computed,
the inertia of \eqref{eq:K2} can be obtained from the block diagonal matrix $B$ (with 1x1 or 2x2 diagonal blocks).
Previous implementations of \gls{ncl}-based ALM~\cite{ma2018stabilized,ma2021julia} have utilized the Newton \gls{kkt} system \eqref{eq:K2} to compute $(\Delta x, \Delta r, -\Delta y)$ during the Newton iterations. Direct sparse linear solvers such as Duff and Reid's MA27 \cite{MA27} and MA57 \cite{MA57} are frequently employed, all utilizing numerical pivoting to ensure numerical stability and to handle ill-conditioned indefinite linear systems.
However, due to the challenges of implementing numerical pivoting on parallel architectures, these traditional direct linear solvers are difficult to port efficiently to GPUs.

\paragraph{Sparse \ldlt factorization}
\ldlt factorization (with $D$ diagonal) is more economical than \lblt factorization, but it is not guaranteed to exist for every symmetric indefinite matrix.
We observe that under a suitable symmetric permutation, \eqref{eq:K2} may admit an \ldlt factorization. In particular, with the permutation matrix
$$
P = \bmat{ 0 & I & 0 \\ 0 & 0 & I \\ I & 0 & 0 },
$$
we can transform \eqref{eq:K2} into block-diagonal form:
\begin{multline}
  \label{eq:ldlk2}
  P^\top
  \bmat{
       \hat{H} + \delta_k I & 0 & J^\top \\
       0       & \hat{\rho}_k I  & I \\
       J       & I         & 0
  } P
  = \\
  \bmat{
    I & 0 & 0 \\
    \theta_k I & I & 0 \\
    0 & -\hat{\rho}_k J^\top & I
  }
  \bmat{
    \hat{\rho}_k I & 0 & 0 \\
    0 & -\theta_k I & 0 \\
    0 & 0 & \hat{H} + \delta_k I + \hat{\rho}_k J^\top J
  }
  \bmat{
    I & \theta_k I & 0 \\
    0 & I & -\hat{\rho}_k J \\
    0 & 0 & I
  } \; ,
\end{multline}
with $\theta_k := \hat{\rho}_k^{-1}$.
This shows that an \ldlt decomposition is always possible
provided the matrix $\hat{H} + \delta_k I + \rho_k J^\tiptop J$ admits itself an \ldlt factorization.
For example, the regularization $\delta_k$ can be sufficiently large. The positive definiteness of $\hat{H} + \delta_k I + \rho_k J^\tiptop J$ plays a crucial role in the inertia analysis in \Cref{sec:inertia}.
However, we emphasize \ldlt decomposition is not guaranteed to exist for every
permutation $P$ because \eqref{eq:K2} is not strongly factorizable, implying that
zero pivots can be encountered during the factorization.

When the \ldlt factorization routine encounters a zero pivot, two cases can
occur: (i) The factorization stops, returning an incorrect inertia to the
nonlinear solver. In this case, the inertia regularization increases
$\delta_k$ until the \ldlt factorization succeeds.
(ii) The factorization replaces a null pivot with a value $\pm \varepsilon$ close to machine precision and proceeds to complete the factorization.
The sign of $\varepsilon$ depends on which block of the matrix (before reordering) the pivot belongs to~\cite{andersen1996}.
Some linear solvers, including cuDSS, do not offer the option to set the sign of $\varepsilon$, reducing control during factorization.
If the linear solver returns the correct inertia, the nonlinear solver does not perform any inertia correction.
However, the factorization using the pivot perturbation strategy decomposes a perturbation of the original
\gls{kkt} system. The descent direction can be recovered afterwards if
iterative refinement is used, with a greater number of refinements needed as
the optimal solution is reached \cite{schenk2006fast,schenk2007matching}.

In the following sections, we present two
alternative reformulations of \gls{kkt} system \eqref{eq:K2} that preserve desirable definiteness and inertia properties.

\subsection{Stabilized KKT system \eqref{eq:K2r}} \label{sec:kkt:k2r}
We condense the Newton system \eqref{eq:K2} to a smaller stabilized KKT system, easier to solve.
Substituting
\begin{equation}\label{eq:K2-elimination}
  \Delta r = \theta_k (\Delta y - y_k + y) - r
\end{equation}
into \ref{eq:K2} gives 
\begin{equation}
  \tag{$K_{2r}$}
  \label{eq:K2r}
  \bmat{
    \hat{H} + \delta_k I & J^\top \\
    J                    & -\theta_k I
  }
  \bmat{
    \phantom{-}\Delta x \\
    -\Delta y
  }
  =
  - \bmat{
    \nabla \phi(x) - J^\top y -(X-L)^{-1}\mu_k e + (U-X)^{-1}\mu_k e \\
    c - \theta_k (y_k-y)
  } \; .
\end{equation}

In contrast to the original system \eqref{eq:K2}, the lower right block in
\eqref{eq:K2r} is negative definite. If problem NLP is strictly convex, the
Hessian in the $(1, 1)$ block of \eqref{eq:K2r} is SPD, and system
\eqref{eq:K2r} is symmetric quasi definite (SQD) without regularization
($\delta_k = 0$). It is known that SQD matrices are strongly factorizable.  Thus in the convex case there exists an \ldlt decomposition for
every symmetric permutation~\cite{Van95}.

If the problem is nonconvex, system \eqref{eq:K2r} is no longer SQD for
$\delta_k = 0$, and the matrix is not strongly factorizable. Nevertheless,
one can show that if a proper permutation is used, the matrix in
\eqref{eq:K2r} still admits a block \ldlt factorization.
For example, like its sibling \eqref{eq:K2}, system \eqref{eq:K2r} can be decomposed after a symmetric permutation using a backward identity matrix as
\begin{equation}
  \begin{aligned}
  \label{eq:ldlk2r}
  \bmat{0 & I \\ I & 0}
  \bmat{
       \hat{H} + \delta_k I & J^\top \\
       J                    & -\theta_k I
  }
  \bmat{0 & I \\ I & 0}
  &=
  \bmat{
    I & 0 \\
    -\hat{\rho}_k J^\top & I
  }
  \bmat{
    -\theta_k I & 0 \\
    0 & \hat{H} + \delta_k I + \hat{\rho}_k J^\top J
  }
  \bmat{
    I & -\hat{\rho}_k J \\
    0 & I
  } \\
  &=
  \bmat{
    L_{11} & 0 \\
    L_{21} & L_{22}
  }
  \bmat{
    D_{11} & 0 \\
    0 & D_{22}
  }
  \bmat{
    L_{11}^\top & L_{21}^\top \\
    0 & L_{22}^\top
  },
  \end{aligned}
\end{equation}
with $L_{11} = I$, $D_{11} = -\theta_k$, $L_{21} = -\hat{\rho}_k J^\top$, and $L_{22} D_{22} L_{22}^\top = \hat{H} + \delta_k I + \hat{\rho}_k J^\tiptop J := M$. Thus, 
the matrix in \eqref{eq:K2r} admits a block \ldlt factorization if the Schur complement $M$ is strongly factorizable (e.g., positive definite).
In \S\ref{sec:inertia} we show that for $\delta_k$
and $\rho_k$ large enough, the matrix $M$ is positive definite.

\subsection{Condensed KKT system \eqref{eq:K1s}}
The stabilized system~\eqref{eq:K2r} can be reduced further by removing the blocks associated with the dual descent direction $\Delta y$.
We obtain a smaller condensed KKT system \eqref{eq:K1s} suitable for a sparse Cholesky factorization, as we 
show in \Cref{sec:inertia}.

System \eqref{eq:K2r} is equivalent to the \emph{condensed \gls{kkt} system}
\begin{align}
\label{eq:K1}
\tag{$K_1$}
(\hat{H} + \delta_k I + \hat{\rho}_k J^\tiptop J)
\Delta x
&=
J^\top (y_k - \hat{\rho}_k c)
- (\nabla \phi(x) - (X-L)^{-1} \mu_k e + (U-X)^{-1} \mu_k e),
\\  \Delta y &=  -\hat{\rho}_k (J \Delta x + c) + y_k - y \;     \label{eq:K1-elimination}.
\end{align}

If NLP represents a problem NCO with inequality constraints, we can reduce further the size of the condensed system by eliminating the slack update $\Delta s$ stored implicitly within $\Delta x$.
We prefer to use~\eqref{eq:K1s} below instead of~\eqref{eq:K1} whenever we have inequality constraints, as the resulting system is smaller.
We introduce the function $h(t) = \bmat{c_{\mathcal{E}}(t)^\top & c_{\mathcal{I}}(t)^\top}^\top$ together
with the matrix $M^\top = \bmat{0 & -I}$ and the two diagonal matrices
\begin{equation}
\Sigma_t = (T - L_t)^{-1} Z_{l, t} + (U_t - T)^{-1} Z_{u, t} \quad \text{and} \quad \Sigma_s = (S - L_s)^{-1} Z_{l, s} + (U_s - S)^{-1} Z_{u,s}.
\end{equation}
The final condensation is detailed in the next proposition.

\begin{proposition}[Condensed KKT system]
  If~\eqref{eqn:nlp} has structure NCO, then~\eqref{eq:K1} is equivalent to
  \begin{equation}
      \tag{$K_{1s}$}
      \label{eq:K1s}
      \left( W_t + \Sigma_t + \hat{\rho}_k J_{c_\mathcal{E}}^\top J_{c_\mathcal{E}} + \hat{\rho}_k J_{c_\mathcal{I}}^\top \Omega_k J_{c_\mathcal{I}} \right) \Delta t = r_t + \hat{\rho}_k J_{c_\mathcal{I}}^\top r_s \; ,
  \end{equation}
  with
  \begin{align}
    W_t &:= \nabla^2 f(t) - \sum_{i=1}^m y_i \nabla^2 h_i(t)  \;,
    \\K_t &:= W_t  + \hat{\rho}_k J_{c_\mathcal{E}}^\top J_{c_\mathcal{E}} + \hat{\rho}_k J_{c_\mathcal{I}}^\top J_{c_\mathcal{I}} + \Sigma_t ,
    \\ \Omega_k &:= \Sigma_s (\hat{\rho}_k + \Sigma_s)^{-1},
    \\ r_t &:= J_h^\top (y_k - \hat{\rho}_k c) - \nabla f(t) + (T-L_t)^{-1} \mu_k e - (U_t-T)^{-1} \mu_k e,
    \\ r_s &:= M^\top (y_k - \hat{\rho}_k c) + (S - L_s)^{-1} \mu_k e - (U_s - S)^{-1} \mu_k e.
  \end{align}
\end{proposition}
\begin{proof}
  As~\eqref{eqn:nlp} has the NCO structure, the
  optimization variable can be decomposed as $x = (t, s)$, with $s$ a slack variable.
  The condensed \gls{kkt} system \eqref{eq:K1} becomes
  \begin{equation}
    \label{eq:K1'}
    \tag{$K_1'$}
    \bmat{
      K_t & -\hat{\rho}_k J_{c_{\mathcal{I}}}^\top
    \\ -\hat{\rho}_k J_{c_{\mathcal{I}}} & \Sigma_s + \hat{\rho}_k I}
    \bmat{
        \Delta t
      \\ \Delta s
    }
    =
    \bmat{
        r_t
      \\ r_s
    }.
  \end{equation}
  We can eliminate
  $\Delta s = (\Sigma_s + \hat{\rho}_k I)^{-1} (\hat{\rho}_k J_{c_{\mathcal{I}}} \Delta t + r_s)$
  in~\eqref{eq:K1'} to obtain the condensed system
  \begin{equation}
    \label{eq:K1s-not-yet}
    \left[K_t - \hat{\rho}_k^2 J_{c_{\mathcal{I}}}^\top (\Sigma_s + \hat{\rho}_k I)^{-1} J_{c_{\mathcal{I}}}\right]
    \Delta t =
    r_t + \hat{\rho}_k J_{c_{\mathcal{I}}}^\top r_s.
  \end{equation}
  The left-hand-side in \eqref{eq:K1s-not-yet} expands as
  \begin{equation*}
  \begin{aligned}
    K_t - \hat{\rho}_k^2 J_{c_{\mathcal{I}}}^\top (\Sigma_s + \hat{\rho}_k I)^{-1} J_{c_{\mathcal{I}}} &=
    W_t + \Sigma_t + \hat{\rho}_k J_{c_\mathcal{E}}^\top J_{c_\mathcal{E}} + \hat{\rho}_k J_{c_\mathcal{I}}^\top(I - \hat{\rho}_k (\hat{\rho}_k + \Sigma_s)^{-1}) J_{c_\mathcal{I}}  , \\
    &=
    W_t + \Sigma_t + \hat{\rho}_k J_{c_\mathcal{E}}^\top J_{c_\mathcal{E}} + \hat{\rho}_k J_{c_\mathcal{I}}^\top\Sigma_s (\hat{\rho}_k + \Sigma_s)^{-1} J_{c_\mathcal{I}}  , \\
    &=
    W_t + \Sigma_t + \hat{\rho}_k J_{c_\mathcal{E}}^\top J_{c_\mathcal{E}} + \hat{\rho}_k J_{c_\mathcal{I}}^\top \Omega_k J_{c_\mathcal{I}},
  \end{aligned}
  \end{equation*}
  where we used the identity
  $I - \hat{\rho}_k (\hat{\rho}_k I + \Sigma_s)^{-1} = \big((\hat{\rho}_k I + \Sigma_s) - \hat{\rho}_k I\big)(\hat{\rho}_k I +\Sigma_s)^{-1} = \Sigma_s (\hat{\rho}_k I + \Sigma_s)^{-1}$.
  By replacing this expression in \eqref{eq:K1s-not-yet}, we obtain \eqref{eq:K1s}.
\end{proof}

On one hand, if an inequality constraint $i$ is active at a solution, then $(\Sigma_s)_{ii} \to +\infty$ as the iterates converge. Hence we have $(\Omega_k)_{ii} \to 1$: the constraint is treated asymptotically as an equality constraint by \gls{ncl}. On the other hand, if the inequality constraint is inactive, then $(\Sigma_s)_{ii} \to 0$, leading to $(\Omega_k)_{ii} \to 0$. The block associated with the inactive constraints becomes negligible in \eqref{eq:K1s}, as expected.

\subsection{Analysis of the inertia of reformulated \gls{kkt} systems}\label{sec:inertia}

\Cref{prop:inertia} shows how the inertia of \eqref{eq:K2} relates to that
of \eqref{eq:K2r} and \eqref{eq:K1s}.
\begin{proposition}
  \label{prop:inertia}
  If $\rho_k, \delta_k >0$, the following statements are equivalent:
  \begin{itemize}
    \item[(i)] $\Inertia(K_2) = (n+m,\, m,\, 0)$;
    \item[(ii)] $\Inertia(K_{2r}) = (n,\, m,\, 0)$;
    \item[(iii)] $\Inertia(K_{1s}) = (n,\, 0,\, 0)$.
  \end{itemize}
\end{proposition}
\begin{proof}
  It is sufficient to show the equivalence of (i) and (ii), as well as that of (ii) and (iii).

  We first show (i) $\iff$ (ii). By Sylvester's law of inertia applied to \ref{eq:K2}, we have
  \begin{align*}
    \Inertia(K_2)
    = \Inertia
      \bmat{\hat{H} + \delta_k I & 0           & J^\top \\
    0       & \hat{\rho}_k I    & I \\
    J       & I           & 0}
    &= \Inertia(\hat{\rho}_k I) + \Inertia\bmat{\hat{H} + \delta_k I & J^\top \\ J & -\theta_k I} \\
    &= (m,0,0) + \Inertia(K_{2r}),
  \end{align*}
  where the third equality follows from the fact that $\hat{\rho}_k > 0$. Therefore, (i) is equivalent to (ii).

  Next, we show (ii) $\iff$ (iii). We apply Sylvester's law of inertia to the matrix in \ref{eq:K2r} to obtain
  \begin{align*}
    \Inertia(K_{2r}) = \Inertia\bmat{\hat{H} + \delta_k I & J^\top \\ J & -\theta_k I}
                     &= \Inertia(-\theta_k I) + \Inertia(\hat{H} + \delta_k I + \hat{\rho}_k J^\top J)\\
                     &= (0,m,0) + \Inertia(K_{1s}).
  \end{align*}
  Thus, $\Inertia(K_{2r}) = (n, m, 0)$ if and only if $\Inertia(K_{1s}) = (n, 0, 0)$, which is equivalent to the matrix $K_{1s}$ being SPD.
\end{proof}
Here, the assumption $\rho_k, \delta_k > 0$ is always satisfied if $\rho_k$ and $\delta_k$ are selected by Algorithm~\ref{algo:nclipm}.

Proposition~\ref{prop:inertia} has important implications for inertia-based IPMs. 
It demonstrates that if either $\Inertia(K_{2r}) = (n,\, m,\, 0)$ or $\hat{H} + \delta_k I + \hat{\rho}_k J^\top J \succ 0$, the solution of 
system~\eqref{eq:K2} provides a search direction for \NCk.
Thus, even if we base the solver on \eqref{eq:K2r} or \eqref{eq:K1s}, their respective inertias determine if the regularization $\delta_k$ is sufficient to ensure that the Newton step is a descent direction.
This guarantees that the solution proceeding with the \eqref{eq:K2r} or \eqref{eq:K1s} systems can achieve, in principle, the same degree of robustness in the convergence behavior as algorithms based on the original system \eqref{eq:K2}.

\section{Implementation}
\label{sec:num:implementation}
We discuss various aspects of our implementation of Algorithm \gls{ncl} and the numerical benchmark studies.

\subsection{MadNCL}
We have implemented Algorithm~\ref{algo:nclipm} in the MadNCL solver, which is publicly available on GitHub.\footnote{\url{https://github.com/MadNLP/MadNCL.jl}} We adapted the Julia implementation {\tt \gls{ncl}}~\cite{ma2021julia} to utilize MadNLP to solve 
the \NCk subproblems.
We leverage MadNLP's warm-start feature to reuse the data structure throughout the iterations, including the initial symbolic factorization. MadNCL employs the {\tt Abstract\gls{kkt}System} abstraction implemented within MadNLP to implement formulations \eqref{eq:K2r} and \eqref{eq:K1s} as {\tt K2rAuglag\gls{kkt}System} and {\tt K1sAuglag\gls{kkt}System} abstractions, respectively. By utilizing the MadNLP formalism, we can efficiently allocate all data structures and perform all computations on the \gls{gpu}, minimizing the transfer of data between the host (CPU memory) and the device (\gls{gpu} memory).

\subsection{Scaling}
The objective of subproblem \NCk combines the original objective $\phi(x)$ with the augmented Lagrangian term $y_k^\top r + \frac{\rho_k}{2}\|r\|^2$. It is therefore sensitive to problem scaling. For instance, if $\rho_k = 1000$, $y_k = 0$, and there are 10,000 constraints of magnitude $O(1)$ while the objective also has a magnitude of $O(1)$, we add an $O(1)$ term to a term of magnitude $O(10^7)$. This can lead to numerical instabilities.

To address this, we employ an automatic scaling strategy similar to that used in Ipopt~\cite{wachter2006implementation}. Given the initial point $x_0$, $g_0 = \nabla f(x_0)$, and $g_i = \nabla c_i(x_0)$ for $i=1,\dots,m$, we scale the objective by $\sigma_f$ and each constraint by $(\sigma_c)_i$, where
\begin{equation}
  \sigma_f = \max\Big\{10^{-8}, \min\Big(1, \frac{\tau}{\|g_0\|_\infty}\Big)\Big\} \;, \quad
  (\sigma_c)_i = \max\Big\{10^{-8}, \min\Big(1, \frac{\tau}{\|g_i\|_\infty}\Big)\Big\} \; ,
\end{equation}
where \(\tau > 0\).
In our implementation we set $\tau = 1$, which differs from the parameter used in Ipopt ($\tau = 100$) but is the same as in ALGENCAN~\cite{birgin2020complexity}.
Despite being more aggressive, this scaling is more appropriate for the ALM.

\subsection{Linear solver}
On the CPU, MadNCL uses the linear solvers HSL MA27 and HSL MA57.
On the GPU, we use the solver NVIDIA cuDSS for the \ldlt factorization in both
\eqref{eq:K2r} and \eqref{eq:K1s}.

On one hand, the condensed system \eqref{eq:K1s} is positive definite after regularization, meaning it is strongly factorizable. We use the default options in cuDSS
for ordering, \emph{pivot threshold} (used to determine if diagonal element is subject to pivoting and will be
swapped with the maximum element in the row or column) and \emph{pivot epsilon} (used to replace the small diagonal elements encountered during numerical factorization).

On the other hand, the stabilized system \eqref{eq:K2r} is not strongly factorizable
(see \S\ref{sec:kkt:k2r}). We use a pivot regularization strategy within cuDSS
by setting the pivot epsilon to $10^{-10}$. As a result, cuDSS returns the factorization of a slightly perturbed matrix.
We recover the solution of the original system by using iterative refinement (Richardson iterations by default) if the inertia is correct.
This strategy has been implemented before in Ipopt, when the KKT systems are solved with the linear solver Panua Pardiso \cite{schenk2007matching}. We plan to investigate more sophisticated strategies
for iterative refinement in the future~\cite{arioli2007note}.

\subsection{Modeler and automatic differentiation}
The benchmark instances are created using ExaModels~\cite{shin2024accelerating}, which allow for automatic differentiation on \gls{gpu}s. In particular, ExaModels employs the SIMD abstraction to optimize the derivative evaluations over embarrassingly parallel objective and constraint function expressions. ExaModels is highly optimized, often rendering the derivative evaluation time a negligible portion of the overall solution time. Consequently, when tested with problem instances implemented using ExaModels, MadNCL's performance is primarily determined by linear algebra computations and solver internals.

\section{Numerical results}\label{sec:num}

We detail the performance of \gls{ncl} (Algorithm~\ref{algo:nclipm}) on various NLP instances.
In \S\ref{sec:num:cutest} we present results on the CUTEst set~\cite{gould2015cutest}
with MadNCL running on the CPU with HSL MA57.
In \S\ref{sec:num:gpu} we analyze
the performance of MadNCL running on the GPU with the linear solver NVIDIA cuDSS on large-scale OPF and COPS instances.
Finally, \S\ref{sec:num:degenerate} shows
the performance of MadNCL on degenerate nonlinear programs arising from power systems.

All results on the CPU have been generated using an AMD EPYC 7443 (24-core) processor.
The benchmarks are generated on the GPU using an NVIDIA H100.
We use labels MadNCL-K2r and MadNCL-K1s to denote
MadNCL computing the Newton descent direction with \eqref{eq:K2r} and \eqref{eq:K1s} respectively.

\subsection{CUTEst benchmark}
\label{sec:num:cutest}

We start by analyzing the performance of MadNCL on problems from the CUTEst collection.
We select all instances with more than 1,000 variables and at least 1 constraint.
We compare MadNCL with Ipopt and MadNLP, 
which both
use the \lblt factorization in HSL MA57, whereas MadNCL uses HSL MA57 with a pivot threshold set to 0.0 to deactivate numerical pivoting.
Results with optimality tolerance {\tt tol=1e-8} are displayed in Figure~\ref{fig:cutest}.
We observe that MadNCL compares favorably with Ipopt and MadNLP. Despite not being
the fastest solver, MadNCL is more robust than MadNLP and Ipopt.
A closer look shows that
MadNCL can solve to optimality instances
where Ipopt and MadNLP are failing because
(i) the problems do not have enough degrees of freedom ({\tt CHAINWOONE}, {\tt NINE5D}, {\tt NINENEW}, {\tt MODBEALENE}, {\tt FIVE20B}, {\tt FIVE20C}, {\tt BDQRTICNE}, {\tt HIER163A}, {\tt HIE1327D}, {\tt HIER133E}, {\tt TWO5IN6})
(ii) the {\tt IPM} algorithm never exits feasibility restoration ({\tt BRAINPC1}, {\tt BRAINPC5}, {\tt BRAINPC9}, {\tt SAROMM}, {\tt ARTIF}), and
(iii) the filter line-search IPM  implemented in Ipopt and MadNLP performs too many primal-dual regularization ({\tt MSS3}, {\tt ORTHREGE}).

In terms of solving time, we observe that MadNCL is particularly effective
at solving the nonconvex QP instances in CUTEst: on {\tt NCVXQP7} and {\tt NCVXQP8}
it takes respectively 32s and 56s compared to 193s and 460s for Ipopt.
Similar speed-ups are reported for other nonconvex QPs ({\tt A0ESSNDL},
{\tt A0NSSSSL}, {\tt A0ENSNDL}, {\tt A0NNSNSL}, {\tt A2ESSNDL}).
However, we observe that MadNCL can be significantly slower on some specific instances
where it converges in 10 times more IPM iterations than Ipopt or MadNLP
({\tt BLOWEYA}, {\tt YAO}, {\tt CHEMRCTA}, {\tt UBH5}, {\tt SEMICON1}, {\tt POROUS1},
{\tt FERRISDC}, {\tt READING4}),
even though we warm-start on each subproblem \NCk except the first.

\begin{figure}[t] 
  \centering
  \includegraphics[width=.8\textwidth]{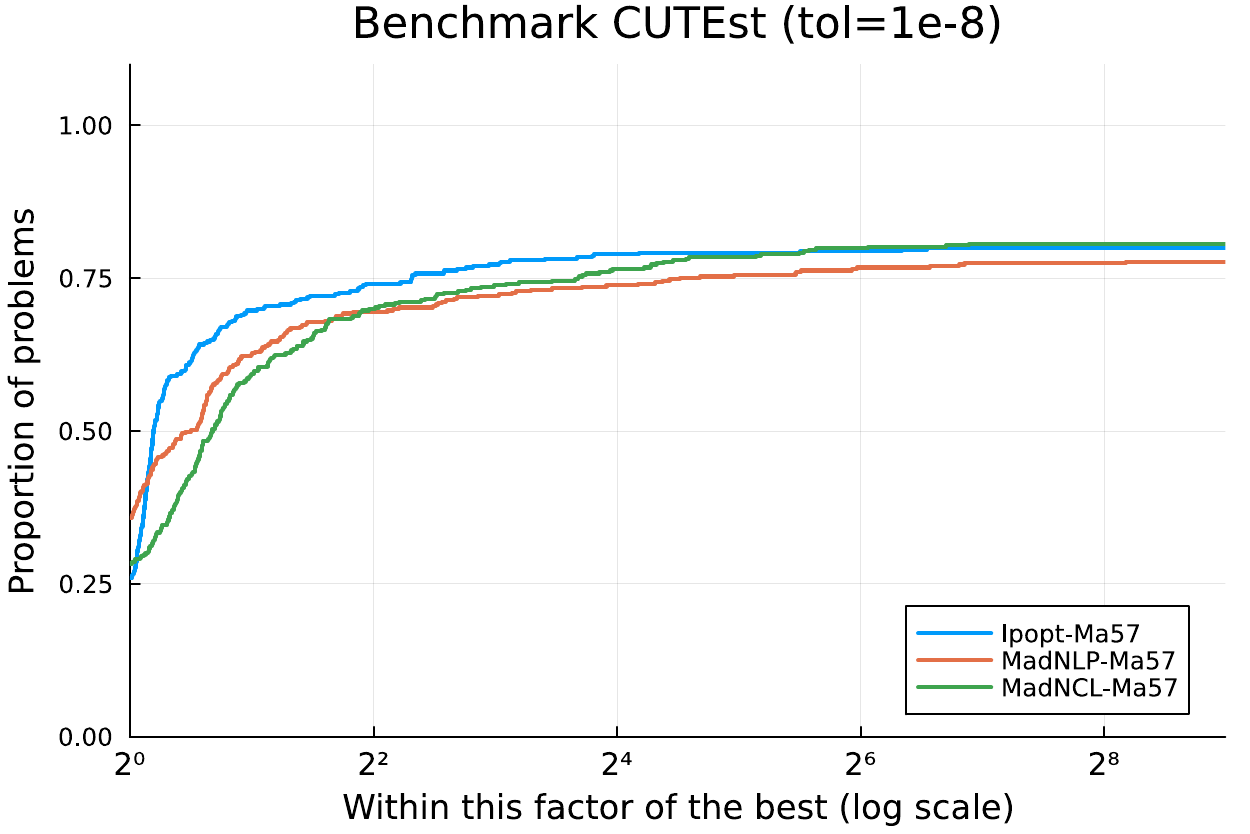}
  \caption{Results on the CUTEst benchmark (using {\tt tol=1e-8}).
    We profile the time (in seconds) to achieve optimality.
    We have selected all instances with at least 1,000 variables and 1 constraint.
    MadNCL is running on the CPU with HSL MA57.
    \label{fig:cutest}
  }
\end{figure}

\subsection{Achieving scalability on \gls{gpu}}
\label{sec:num:gpu}

The CUTEst benchmark presented in \S\ref{sec:num:cutest} runs entirely on the CPU.
In this subsection, we assess the performance achieved by MadNCL on the \gls{gpu}.
To do so, we move the models evaluation to the GPU using the modeler ExaModels.
The linear systems \eqref{eq:K2r} and \eqref{eq:K1s} are factorized on
the GPU using NVIDIA cuDSS.
The resulting algorithm runs entirely on the GPU, from the evaluation of the
derivatives to the computation of the Newton step. This avoids costly transfers
between the host and the device memory.
Our GPU-accelerated benchmark includes OPF instance from PGLIB~\cite{babaeinejadsarookolaee2019power}
and large-scale nonlinear programs from COPS~\cite{dolan2004benchmarking}.
Unfortunately, CUTEst does not support evaluating the models on the GPU, in constrast to ExaModels.

\subsubsection{Assessing MadNCL's performance on the \gls{gpu}}
\label{sec:num:gpu:opf}
We assess the speed-up obtained when MadNCL solves on the GPU the large \gls{opf} instance {\tt 78484epigrids}, compared to a CPU implementation. The instance has 674,562 variables and 1,039,062 constraints. We set the tolerance to {\tt tol=1e-8}. On the CPU we use the linear solver HSL MA27 (which is faster than MA57 for OPF problems~\cite{tasseff2019exploring}), and on the GPU we use NVIDIA cuDSS.

Results from the largest instance {\tt 78484epigrids} show that the linear
solver cuDSS is effective when using the formulation \eqref{eq:K2r}: MadNCL converges in 330 IPM iterations in
54.4s, an 18x speed-up compared to MadNCL on the CPU with HSL MA27.
For comparison, the reference method MadNLP (using HSL MA27) converges in 328 seconds to the same
solution.
The performance difference observed between CPU and \gls{gpu} is largely
explained by the faster solution time in the linear solver, with cuDSS being
effective at factorizing \eqref{eq:K2r} on the \gls{gpu}. If we report the time
per iteration, cuDSS is 17 times faster than HSL MA27 at factorizing system \eqref{eq:K2r}.

MadNCL fails to converge when using the $K_{1s}$ formulation. The algorithm struggles to address the increased ill-conditioning in the condensed KKT system~\eqref{eq:K1s} in this case. (We see in the next section that this is not always the case.)  Overall, \eqref{eq:K1s} proves to be significantly less robust than the formulation \eqref{eq:K2r}.
A closer investigation has shown that \eqref{eq:K1s} is not able to find an
accurate descent direction in the final IPM iterations (despite the iterative refinement
we are using), impairing MadNCL's convergence.
As $\rho_k$ increases, the matrix \eqref{eq:K1s} becomes too ill-conditioned
and the descent direction is not sufficiently accurate.

\begin{table}[t] 
  \caption{Performance comparison of MadNCL on a large-scale \gls{opf} instance
  using different formulations for the KKT systems. As a baseline, we give
  the time spent in MadNLP on the CPU in the first row. The columns are
  (i) {\bf flag}: return status for MadNLP (1: locally optimal, 0: failed to converge);
  (ii) {\bf it}: total number of IPM iterations;
  (iii) {\bf lin}: time spent in the linear solver (in seconds);
  (iv) {\bf total}: time spent in the solver (in seconds).
  }
  \centering
  \begin{tabular}{|lll|rrr>{\bfseries}r|}
    \toprule
  && &
  \multicolumn{4}{c|}{\bf 78484epigrids} \\
  \midrule
    {\bf solver} & KKT      & linear solver & flag & it   & lin    & total  \\
    \midrule
    MadNLP       & $K_2$    & ma27      & 1    & 104  & 312.6  & 353.9  \\
    \midrule
    MadNCL       & $K_2$    & ma27      & 1    & 322  & 1053.6 & 1182.0 \\
    MadNCL       & $K_{2r}$ & ma27      & 1    & 322  & 847.9  & 971.5  \\
    MadNCL       & $K_{2r}$ & cudss-ldl & 1    & 330  & 50.1   & 54.4   \\
    MadNCL       & $K_{1s}$ & ma27      & 0    & 1000 & 3222.2 & 3848.6 \\
    MadNCL       & $K_{1s}$ & cudss-ldl & 0    & 1000 & 154.1  & 170.2  \\
    \bottomrule
  \end{tabular}
\end{table}

\subsubsection{Comparing MadNCL with other GPU-accelerated solvers}

In this section, we present our principal benchmark on
large-scale \gls{opf} and COPS instances~\cite{babaeinejadsarookolaee2019power,dolan2004benchmarking}:
we compare MadNCL running on the GPU with
MadNLP (running on the CPU using HSL MA27, and on the GPU using the two GPU-accelerated KKT linear solvers
HyKKT and LiftedKKT \cite{pacaud2024condensed}).
The results are presented in Table~\ref{tab:benchgpu}. We observe that
MadNCL exhibits better performance on the COPS instances than on the OPF ones.

The OPF instances have very sparse Jacobians.
We observe that MadNCL-K2r is able to solve
all cases except {\tt 10480\_goc}.
This is much better than MadNCL-K1s, which fails on all instances.
MadNCL-K2r requires more IPM iterations than MadNLP, as expected
for an Augmented Lagrangian method (superlinear convergence is achieved
only when we enter the extrapolation step). However, the GPU-acceleration benefits MadNCL-K2r, which is overall faster than MadNLP running on the CPU with HSL MA27:
MadNCL-K2r achieves a 6x speed-up on the largest instance {\tt 78484\_epigrids}.
On the other hand, MadNCL-K2r remains slower than MadNLP running on the GPU, with LiftedKKT being the
fastest method here.

Results on the COPS instances are different from the OPF instances. In contrast to what we observed earlier, both MadNCL-K2r and MadNCL-K1s exhibit
reliable convergence and are significantly faster than MadNLP on the CPU
with HSL MA27. MadNCL-K1s is here the fastest NCL variant, and compares favorably w.r.t.\ MadNLP running on the GPU
with HyKKT and LiftedKKT. In contrast to the OPF instances, the number
of IPM iterations in MadNCL does not increase significantly compared to MadNLP.
The failure observed on the {\tt rocket} instance is noteworthy, as this problem is known to exhibit a specific type of degeneracy: the reduced Hessian is nearly singular at optimality.
In that situation, we have no guarantee that the Augmented Lagrangian method used
in MadNCL would converge~\cite{fernandez2012local}.
This clarifies that the issue arises from a different form of degeneracy than that handled by NCL.

\begin{table}[t]
\caption{Comparing the performance of MadNLP and MadNCL on the \gls{gpu} for large-scale
  \gls{opf} instances from PGLIB~\cite{babaeinejadsarookolaee2019power} and large-scale COPS instances~\cite{dolan2004benchmarking}.
  The tolerance is {\tt tol=1e-8}.
  The columns are:
  (i) {\bf flag}: return status for MadNLP (1: locally optimal, 2: solved to acceptable level, 0: failed to converge);
  (ii) {\bf it}: total number of IPM iterations;
  (iii) {\bf lin}: time spent in the linear solver (in seconds);
  (iv) {\bf total}: time spent in the solver (in seconds).}
  \label{tab:benchgpu}
  \centering
  \resizebox{\textwidth}{!}{
    \begin{tabular}{|l|rrr >{\bfseries}r|rrr >{\bfseries}r|rrr >{\bfseries}r|rrr >{\bfseries}r|rrr >{\bfseries}r|}
      \toprule
      & \multicolumn{4}{c|}{\bf MadNLP-K2-Ma27} &
      \multicolumn{4}{c|}{\bf MadNLP-HyKKT-cuDSS} &
      \multicolumn{4}{c|}{\bf MadNLP-LiftedKKT-cuDSS} &
      \multicolumn{4}{c|}{\bf MadNCL-K2r-cuDSS} &
      \multicolumn{4}{c|}{\bf MadNCL-K1s-cuDSS} \\
      \midrule
      \textbf{case} & \textbf{flag} & \textbf{it} & \textbf{lin} & \textbf{total} & \textbf{flag} & \textbf{it} & \textbf{lin} & \textbf{total} & \textbf{flag} & \textbf{it} & \textbf{lin} & \textbf{total} & \textbf{flag} & \textbf{it} & \textbf{lin} & \textbf{total} & \textbf{flag} & \textbf{it} & \textbf{lin} & \textbf{total} \\
      \midrule
      9241\_pegase    & 1 & 75   & 7.04    & 9.77    & 1  & 72   & 0.88   & 1.75   & 0 & 1000 & 11.85 & 22.50 & 1 & 301  & 6.25  & 8.29  & 0 & 835  & 43.81  & 51.91  \\
      9591\_goc       & 1 & 68   & 12.92   & 15.52   & 1  & 194  & 1.75   & 3.88   & 1 & 85   & 1.50  & 2.33  & 1 & 89   & 1.73  & 2.74  & 0 & 1000 & 47.00  & 57.76  \\
      10000\_goc      & 1 & 85   & 7.37    & 10.20   & 1  & 85   & 0.58   & 1.35   & 1 & 65   & 0.73  & 1.33  & 1 & 108  & 1.65  & 2.43  & 0 & 1000 & 27.36  & 35.36  \\
      10192\_epigrids & 1 & 56   & 9.88    & 12.78   & 2  & 69   & 0.74   & 2.15   & 1 & 58   & 1.28  & 2.15  & 1 & 143  & 2.37  & 3.37  & 0 & 1000 & 71.93  & 81.12  \\
      10480\_goc      & 1 & 72   & 14.21   & 17.52   & 1  & 340  & 4.42   & 7.92   & 1 & 68   & 0.91  & 1.75  & 0 & 1000 & 40.63 & 48.21 & 0 & 1000 & 62.12  & 71.12  \\
      13659\_pegase   & 1 & 65   & 12.48   & 16.44   & 1  & 64   & 0.70   & 1.75   & 1 & 72   & 0.85  & 1.89  & 1 & 254  & 6.18  & 8.20  & 0 & 1000 & 35.31  & 44.17  \\
      19402\_goc      & 1 & 72   & 54.93   & 62.31   & 1  & 541  & 9.26   & 14.71  & 1 & 72   & 1.21  & 3.12  & 1 & 163  & 5.97  & 7.17  & 0 & 657  & 55.65  & 63.24  \\
      20758\_epigrids & 1 & 53   & 27.81   & 33.31   & 1  & 55   & 0.88   & 2.78   & 0 & 1000 & 26.26 & 34.82 & 1 & 284  & 12.85 & 15.04 & 0 & 1000 & 93.64  & 104.63 \\
      24464\_goc      & 1 & 64   & 40.26   & 48.12   & 2  & 609  & 13.71  & 20.39  & 1 & 65   & 1.05  & 3.03  & 1 & 554  & 32.80 & 36.87 & 0 & 945  & 58.38  & 68.17  \\
      30000\_goc      & 1 & 229  & 159.86  & 182.19  & 1  & 231  & 3.76   & 5.92   & 1 & 153  & 2.36  & 3.91  & 1 & 343  & 19.34 & 22.06 & 0 & 1000 & 61.73  & 72.21  \\
      78484\_epigrids & 1 & 104  & 312.58  & 353.85  & 0  & 1000 & 42.39  & 55.47  & 1 & 105  & 4.58  & 10.10 & 1 & 330  & 50.10 & 54.39 & 0 & 1000 & 154.14 & 170.23 \\
      \midrule
      bearing         & 1 & 18   & 13.35   & 22.61   & 1  & 18   & 0.76   & 3.55   & 1 & 15   & 0.12  & 2.86  & 1 & 11   & 0.18  & 3.36  & 1 & 11   & 0.16   & 1.41   \\
      catmix          & 1 & 20   & 1.30    & 7.44    & 1  & 18   & 0.05   & 2.39   & 1 & 26   & 0.18  & 2.51  & 1 & 114  & 0.70  & 3.75  & 1 & 39   & 0.19   & 1.58   \\
      channel         & 1 & 6    & 1.09    & 7.29    & 1  & 6    & 0.03   & 2.74   & 1 & 7    & 0.03  & 2.61  & 1 & 24   & 0.28  & 1.90  & 1 & 21   & 0.14   & 1.36   \\
      elec            & 1 & 209  & 98.03   & 125.83  & 1  & 103  & 0.69   & 4.26   & 1 & 237  & 1.71  & 7.02  & 1 & 199  & 3.33  & 8.01  & 1 & 141  & 1.00   & 3.07   \\
      gasoil          & 1 & 20   & 1.47    & 16.90   & 1  & 21   & 0.13   & 3.19   & 1 & 43   & 0.86  & 4.93  & 1 & 18   & 0.28  & 2.81  & 1 & 18   & 0.14   & 2.11   \\
      marine          & 1 & 14   & 2.19    & 9.45    & 1  & 11   & 0.07   & 2.81   & 1 & 25   & 0.35  & 3.14  & 1 & 22   & 0.23  & 1.97  & 1 & 22   & 0.15   & 1.59   \\
      pinene          & 1 & 12   & 2.40    & 10.60   & 1  & 13   & 0.15   & 1.13   & 1 & 19   & 0.27  & 1.38  & 1 & 51   & 0.74  & 0.98  & 1 & 51   & 0.41   & 1.53   \\
      polygon         & 1 & 32   & 0.02    & 0.03    & 1  & 31   & 0.06   & 0.20   & 1 & 189  & 0.30  & 1.03  & 1 & 67   & 0.17  & 0.53  & 1 & 67   & 0.25   & 0.58   \\
      robot           & 1 & 31   & 1.52    & 11.81   & 1  & 87   & 162.30 & 167.62 & 1 & 26   & 0.06  & 3.35  & 1 & 53   & 0.30  & 6.49  & 1 & 67   & 0.36   & 3.68   \\
      rocket          & 1 & 75   & 1.11    & 10.56   & 1  & 183  & 21.27  & 27.10  & 1 & 111  & 0.32  & 4.46  & 0 & 140  & 0.53  & 4.93  & 1 & 183  & 0.42   & 3.66   \\
      steering        & 1 & 17   & 0.18    & 7.79    & 1  & 16   & 0.04   & 2.56   & 1 & 14   & 0.06  & 2.43  & 1 & 15   & 0.07  & 3.11  & 1 & 15   & 0.06   & 1.69   \\
      torsion         & 1 & 14   & 0.74    & 0.96    & 1  & 14   & 0.08   & 0.29   & 1 & 15   & 0.05  & 0.25  & 1 & 12   & 0.09  & 0.16  & 1 & 12   & 0.10   & 0.17   \\
      \bottomrule
    \end{tabular}
  }
\end{table}

\subsection{\glspl{nlp} with degenerate constraints}
\label{sec:num:degenerate}

The previous benchmark in \S\ref{sec:num:gpu} focused  on solving regular nonlinear programs. Now
we assess how MadNCL performs on degenerate \glspl{nlp}.
We consider practical degenerate instances arising in power system
applications: the security-constrained optimal power flow (SCOPF),
here formulated as mathematical programs with complementarity constraints (MPCC).
The detailed formulation can be found in \cite{aravena2023recent}.
Here, the problem associates each contingency scenario with a given line failure.
The total number of contingencies is denoted by $n_K$: the number of complementarity
constraints increases linearly w.r.t.\ $n_K$, rendering the problem more difficult to solve
for large $n_K$.
By nature, MPCCs are degenerate: the Mangasarian-Fromovitz constraint qualification
(MFCQ) is violated at every feasible point.
Despite this degeneracy, it has been proven that the Augmented Lagrangian method can converge (globally)
to a strongly stationary solution, provided a regularity condition holds at the solution~\cite{izmailov2012global}.
The SCOPF instances here are large-scale and degenerate, offering a good use
case to test MadNCL's performance. As these instances are harder to solve, we relax
the tolerance to {\tt tol=1e-5}.

Classical nonlinear IPM solvers like Ipopt and MadNLP
fail to solve the SCOPF instances to optimality: both solvers suffer from the lack of problem regularity
and they converge to an infeasible solution after entering feasibility restoration.
In contrast, MadNCL converges reliably despite
the ill-conditioning of linear systems \eqref{eq:K2r} and \eqref{eq:K1s}.
The results are displayed in Table~\ref{tab:scopf}.
MadNCL-K2r is able to solve all instances to {\tt tol=1e-5} on the CPU (using HSL MA57) and on the GPU (using cuDSS).
As in \S\ref{sec:num:gpu:opf}, we observe that MadNCL-K2r is more robust than MadNCL-K1s, which suffers from the increased ill-conditioning of system \eqref{eq:K1s}.
In addition, MadNCL-K2r is significantly faster when using GPU-acceleration, with up to
a 10x speed-up on the largest instances.

\begin{table}[t]
\caption{Comparing the performance of MadNCL on large-scale SCOPF instances.
  The tolerance is {\tt tol=1e-5}.
  Column $n_K$ shows the number of contingencies used in the SCOPF.
  Columns $n$ and $m$ show the number of variables and constraints.
  The other columns are
  (i) {\bf flag}: return status for MadNLP (1: locally optimal, 2: solved to acceptable level, 0: failed to converge);
  (ii) {\bf it}: total number of IPM iterations;
  (iii) {\bf lin}: time spent in the linear solver (in seconds);
  (iv) {\bf total}: time spent in the solver (in seconds).}
  \label{tab:scopf}
  \centering
  \resizebox{\textwidth}{!}{
    \begin{tabular}{|lrrr|rrr >{\bfseries}r|rrr >{\bfseries}r|rrr >{\bfseries}r|rrr >{\bfseries}r|}
      \toprule
      & &&& \multicolumn{4}{c|}{\bf MadNCL-K2r-Ma57} & \multicolumn{4}{c|}{\bf MadNCL-K1s-MA57} & \multicolumn{4}{c|}{\bf MadNCL-K2r-cuDSS} & \multicolumn{4}{c|}{\bf MadNCL-K1s-cuDSS} \\
      \textbf{case} & $n_K$ & $n$ & $m$ &\textbf{flag} & \textbf{it} & \textbf{lin} & \textbf{total} & \textbf{flag} & \textbf{it} & \textbf{lin} & \textbf{total} & \textbf{flag} & \textbf{it} & \textbf{lin} & \textbf{total} & \textbf{flag} & \textbf{it} & \textbf{lin} & \textbf{total} \\
      \midrule
      118        & 100 & 131588 & 168853 & 1 & 25  & 3.25   & 7.90   & 1 & 25  & 2.79   & 8.90   & 1 & 25  & 0.62  & 0.89  & 1 & 25  & 1.54  & 5.08  \\
      300        & 100 & 268282 & 350967 & 1 & 49  & 13.23  & 19.44  & 1 & 49  & 9.94   & 16.15  & 1 & 49  & 1.95  & 2.64  & 1 & 49  & 1.29  & 1.75  \\
      ACTIVSg200 & 100 & 162356 & 211571 & 1 & 31  & 8.63   & 14.17  & 1 & 31  & 6.11   & 11.20  & 1 & 33  & 1.19  & 3.63  & 1 & 33  & 0.71  & 2.43  \\
      1354pegase & 8 & 109056 & 144327 & 1 & 42  & 13.38  & 16.44  & 1 & 42  & 11.51  & 16.43  & 1 & 45  & 1.41  & 3.62  & 0 & 250 & 12.94 & 17.08 \\
      1354pegase & 16 & 206920 & 273999 & 1 & 42  & 31.83  & 36.76  & 1 & 42  & 31.82  & 37.20  & 1 & 46  & 2.90  & 3.29  & 1 & 42  & 1.90  & 2.28  \\
      1354pegase & 32 & 402648 & 533343 & 1 & 165 & 202.25 & 234.95 & 0 & 250 & 447.04 & 499.25 & 1 & 218 & 54.42 & 56.74 & 1 & 248 & 35.51 & 39.86 \\
      2869pegase & 8 & 242102 & 323479 & 1 & 45  & 25.60  & 31.32  & 0 & 250 & 182.80 & 212.34 & 1 & 45  & 3.80  & 4.22  & 0 & 250 & 25.95 & 28.82 \\
      2869pegase & 16 & 459118 & 613727 & 1 & 48  & 63.34  & 75.21  & 1 & 56  & 85.30  & 98.42  & 1 & 50  & 8.32  & 8.90  & 0 & 250 & 52.15 & 55.71 \\
      \bottomrule
    \end{tabular}
  }
\end{table}

\section{Conclusions and future work}\label{sec:conc}
We have explored a \gls{gpu} implementation of Algorithm \gls{ncl} for solving large-scale NLP problems,
especially ones whose constraints fail LICQ at a solution.
NCL's need for an IPM subproblem solver is provided by MadNLP.
We focused on some
limitations of existing \gls{gpu}-accelerated NLP solvers.
We demonstrated that problem structure within MadNLP can be leveraged at the linear algebra level by introducing the stabilized KKT system \ref{eq:K2r} and the condensed KKT system \ref{eq:K1s}, which facilitate efficient linear algebra computation on \gls{gpu}s by utilizing MadNLP's flexible abstraction of \gls{kkt} systems.
From a computational standpoint, we have established the first \gls{gpu}-friendly augmented system that
avoids explicit formation of a Schur complement, offering significant advantages for large problems while remaining general enough to address a broad class of problems.
Extensive numerical experiments have validated the performance of MadNCL on \gls{gpu}s, showcasing it as a reliable solver for 
optimization problems
regardless of their degeneracy.

MadNCL holds great potential for future research in addressing challenges such as
failure of LICQ at a solution,
absence of strict complementarity,
or the presence of complementarity constraints (MPCC)
as in the SCOPF examples studied
in \S\ref{sec:num:degenerate}.
Incorporating outer/inner loop structures inspired by \gls{ncl} into existing frameworks like Superb~\cite{gould2015interior} could broaden the applicability of this approach.


\end{document}